\documentclass[letterpaper,conference]{ieeeconf}

\usepackage{amsmath,amsfonts,theorem}
\usepackage{graphicx,epsfig,color}
\usepackage{algorithmic}

\newtheorem{theorem}{Theorem}[section]

\newtheorem{definition}[theorem]{Definition}
\newtheorem{lemma}[theorem]{Lemma}

\newtheorem{assumption}{Assumption}[section]
{\theorembodyfont{\rmfamily} 
\newtheorem{remark}[theorem]{Remark}

}

\newcommand{\myindent}{\hspace{1.5em}}

\newcommand{\Prob}{\operatorname{Prob}}
\newcommand{\real}{\mathbb{R}}

\newcommand{\until}[1]{\{1,\dots, #1\}}

\newcommand{\setdef}[2]{\{#1 \; | \; #2\}}

\newcommand{\map}[3]{#1: #2 \rightarrow #3}
\newcommand{\union}{\operatorname{\cup}}

\newcommand\oprocendsymbol{\hbox{$\square$}}
\newcommand\oprocend{\relax\ifmmode\else\unskip\hfill\fi\oprocendsymbol}
\def\eqoprocend{\tag*{$\square$}}

\title{Algorithms for regional source localization}

\author{Sandra H. Dandach\thanks{The authors are with the Center for
    Control, Dynamical Systems and Computation, University of California at
    Santa Barbara, Santa Barbara, CA 93106, USA,
    \texttt{\{sandra,bullo\}@engineering.ucsb.edu}. This work was supported
    in part by Award AFOSR FA9550-07-1-0528.}  \qquad Francesco Bullo }

\renewcommand{\baselinestretch}{.92}

\begin{document}
\maketitle%
\begin{abstract}
  In this paper we use the MAP criterion to locate a region containing a
  source.  Sensors placed in a field of interest divide the latter into
  smaller regions and take measurements that are transmitted over noisy
  wireless channels. 
  We propose implementations of
  our algorithm that consider complete and limited communication among
  sensors and seek to choose the most likely hypothesis. Each hypothesis
  corresponds to the event that a given region contains the source.
  Corrupted measurements are used to calculate conditional posteriors. 
We prove that the algorithms asymptotically find the correct
  region almost surely as long as information is available from three or
  more sensors. We also study the geometric properties of the model that
  make it possible in some situations to detect the correct region with a
  unique sensor. Our simulations confirm that the performance of algorithms
  with complete and limited information ameliorates with decreasing noise.
\end{abstract}
\section{Introduction} \label{sec_int}
\subsection{Problem description and motivation}
Source localization has assumed increasing interest, and has been the
subject of study for many researchers. The general setting is that one or
multiple sources lie in a bounded region $C$, and a group of $N$ sensors divide
$C$ into $N$ smaller connected localization regions $W_i$ ,where $i \in \{1,\dots,N\}$. They
measure a received signal strength originating at a source $s$, the
sensors try to cooperatively identify the region $W_i$ containing $s$.
%
We set the problem as a multiple hypothesis decision making problem, where
hypothesis $H_i$ is true if the source lies in the region $W_i$. Maximum a
posteriori estimation (MAP) is used as a decision making tool. We implement
the estimation technique with an all-to-all and a limited communication
algorithm.  The setting of the problem and the proposed solution prove to
have some geometric characteristics that we derive later in the paper.
If properly exploited, these characteristics imply the
possibility of regional localization with a unique sensor for certain
source positions. We also prove almost sure $(a.s.)$ convergence of both our
all-to-all and limited communication algorithms. To the best of our
knowledge, none of the algorithms in the literature provide a similar
convergence result.
\subsection{Literature review}
In the classical setting, a number of sensors collaborate to locate the
exact position of a source. The relation between the position of a source
and the received signal strength (RSS) is described in
\cite{TR:96,JP-MS:01,AS-AT-K:05}. RSS indirectly provides the distance
between the source and a sensor. It is easy to formulate a trivial linear
algorithm that permits localization from the measurements at three sensors.
However, such a linear algorithm
may deliver highly inaccurate estimates of the distances, even when the noise
is small \cite{MC-BA-AM:06}. On the other hand, several authors~\cite{AH-DB:05,MR-RN:04}
treat localization as a nonconvex optimization problem. Gradient descent
algorithms can be used to solve the maximum likelihood estimation problems.
Other approaches include approximating the
nonlinear nonconvex optimization problem by a linear and convex one and
then proposing algorithms for the relaxed problem \cite{CM-ZD-SD:08}.

Following the gradient, and approximating with a linear convex problem have
limitations. The gradient descent can get stuck at local minima far from
the correct position, leading to the choice of wrong regions, even in the
absence of noise \cite{GM-BF-BA:07}.  
Authors in \cite{AH-DB:05}
use a method of projection onto convex sets. A necessary and sufficient
condition for the convergence of this algorithm is that the source lies
inside the convex hull of the sensors.
The limitations of these methods motivated us to look into regional localization.
We note that there
are instances where the location of the region containing a source is all
that is needed. A more detailed listing of the contributions of this paper
is presented below.
\subsection{Contributions}
This paper presents the source localization problem in a setting and
formulation that to the best of our knowledge are new. We present
algorithms based on all-to-all and limited communication that require only
the computation of integrals and therefore present a less computationally
exhaustive alternative to the current solutions to the localization
problem. We also show that as the noise decreases, regional localization can be
accomplished with a unique sensor for certain source positions. We show
through an asymptotic analysis that choosing Voronoi partitions as
localization regions achieves zero probability of error in the two sensors
case. The most important advantage of our formulation is that we are able
to demonstrate the convergence of our algorithms. We provide the proof of
$a.s.$ convergence of our algorithms, a step that tends to be missing
in all of the work presented earlier. Finally, the limited communication
algorithm is promising for the localization problems involving multiple
sources.
\subsection{Paper organization}
The paper proceeds by a problem formulation and an explanation of our
proposed solution in Section~\ref{sec_prob_form}.  In
Section~\ref{sec_properties} we derive some asymptotic geometric properties
of the MAP algorithm when applied to our setting. Section~\ref{sec_map}
introduces the implementation of the algorithms. The analytical proof of
almost sure $(a.s.)$ convergence of our algorithms is presented in
Section~\ref{sec_conv}.  Section \ref{sec_simulations} shows our simulation
results and we conclude in Section \ref{sec_conc}.
\section{Problem formulation}\label{sec_prob_form}
Consider a compact connected environment $C\subset\real^2$.  Suppose there
are $N$ disjoint regions $W_i$, such that $\cup_{i=1}^{N} W_i = C$.
Suppose also that there are $N$ sensors placed at $x_i \in W_i$ and that
the source located at an unknown location $s \in C$, transmits a signal
whose power undergoes lognormal shadowing described below.

The average power loss for an arbitrary Transmitter-Receiver separation is
expressed as a function of distance by using a path loss exponent
$\beta>2$. The power loss is proportional to a power of the distance
between the transmitter and the receiver. For a thorough description of
signal attenuation models over communication channels, we refer the reader
to \cite{TR:96,JP-MS:01,CM-ZD-SD:08}. For reasons to be explained shortly, we
work with a slight modification of the traditional model used in
the literature. This model for the received power at a
sensor $i$ is, $ P_{r_i}=\frac{P d_0}{d_0+\|x_i-s\|^{\beta}}$, where
$\beta$ indicates the rate at which the power loss increases with distance.
$d_0$ is a nominal distance chosen such that the received power in the
vicinity of the source is almost equal to $P$, the transmitted power at the
source.  Note that while this model gets rid of the singularity at the
source, it converges to the same behavior as the classical model used in
communication literature $P_{r_i}=\frac{P_1}{\|x_i-s\|^{\beta}}$, when the
distance $\|x_i-s\|$ is large. Here $P_1$ is the power received at a unit
distance from the source. Taking noise into account in our model, the
received power satisfies
\begin{equation}\label{eq-log-normal}
  \ln P_{r_i} = \ln(P d_0) - \ln (d_0+\|x_i - s \|^{\beta}) +
  n_i,
\end{equation}
where $n_i$ are zero mean, independent and identically distributed (i.i.d)
white gaussian noise with variance $\sigma^2$, each associated with a
sensor $i$. The joint probability density function of the $\ln P_{r}=[\ln
P_{r_1},\dots,\ln P_{r_N}]^T$, conditioned on the hypothesis that the
source is at a point $y \in C$ is given by
\begin{multline} \label{eq-joint_cond}
  p(\ln P_{r_1},\dots,\ln P_{r_N}|y) = \frac{1}{(2 \pi \sigma^2)^{N/2}} \\
  \cdot \exp\!\Big({-\frac{\sum_{i=1}^{N}\left(\ln P_{r_i}-\ln(\frac{P
          d_0}{d_0+\|x_i-y\|^{\beta}})\right)^2}{2 \sigma^2}}\Big).
\end{multline}
Solving for the exact position of the source requires solving for $\hat{y}$
that will maximize the likelihood of having the received observation
which becomes the problem of solving for,
\begin{equation*}
  \hat{y} = \arg\min_y \sum_{i=1}^{N} \left(\ln P_{r_i} - \ln (\frac{P
      d_0}{d_0+||x_i-y||^{\beta}})\right)^2.
\end{equation*}
This is a nonlinear nonconvex optimization problem. Attempts to solve it or
approximate its solution are a topic of great interest. In this paper we
look for a regional localization, so the conditioning on the exact position
$y$ in \eqref{eq-joint_cond} is replaced by a regional conditioning. 
The information exchanged between any two
communicating sensors are: the position of the sensors, the localization
regions associated with each sensor and the logarithms of the received
powers (corrupted with log-normal noise).
Sensors can share information as soon as they make a measurement.
Alternatively when the noise level in the communication channel is known to
be high, it is possible for each sensor to average a set of repeated
measurements and transmit the averaged logarithm of the received power.
Averaging helps decrease the noise variance, and therefore as we expect and
will show later, improves the performance. We start by introducing the case
of one noisy measurement per sensor.
\subsection{Posterior density with a single noisy measurement}
Since we do not know where the source is, we make a worst case assumption
on the knowledge of its position $s$. Specifically, we assume that the
density of $s$ obeys:
\begin{equation*}
  p(s) =
  \begin{cases}
    1/A, & \text{if } s \in C, \\
    0,   & \text{otherwise.}
  \end{cases}
\end{equation*}
Here $A$ is the sum of all the areas $A_j$ of $W_j$, with $j \in \{1,\dots,N\}$.
We need to derive the probability density conditioned on each hypothesis.
\vspace{-0.05in}
\begin{lemma}[Regional conditional density]
  Let $z = [\ln P_{r_1},\dots,\ln P_{r_i}]^{T}$, and note that $P(y \in
  W_j)=P(H_j)=\frac{A_j}{A}$, then
  \begin{equation}\label{eq-regional-cond}
    p(z|y \in W_j) = \frac{1}{A_j} \int_{W_j} p(z|y)dy.
  \end{equation}
\end{lemma}
\begin{proof}
  We compute
  \begin{multline*}
    p(z|y \in W_j) = \frac{d}{dz} \frac{\Prob(Z \leq z,y \in
      W_j)}{P(y
      \in W_j)} \\
    = A\frac{d}{dz}\frac{\int_{-\infty}^{z} \qquad \int_{W_j}
      p(z|y)p(y) dy dz}{A_j} \\
    = A\frac{d}{dz}\frac{\int_{-\infty}^{z} \qquad \int_{W_j}
      \big(p(z|y)/A) dy dz}{A_j}
    = \frac{\int_{W_j}p(z|y)dy}{A_j}.
  \end{multline*}
\end{proof}
From~\eqref{eq-joint_cond} and~\eqref{eq-regional-cond}, we obtain
\begin{multline}
  \label{eq-joint_cond_region}
  p(\ln P_{r_1},\dots,\ln P_{r_N}|H_j) P(H_j) = \frac{1}{A} \int_{W_j}
  \frac{1}{(2 \pi
    \sigma^2)^{N/2}} \\
  \cdot \exp\!\Big(-\frac{\sum_{i=1}^{N}\left(\ln P_{r_i}-\ln(\frac{P
        d_0}{d_0+\|x_i-y\|^{\beta}})\right)^2}{2 \sigma^2}\Big) dy.
\end{multline}
Similarly, when measurements from only one sensor are studied.
\subsection{Posterior density with aggregated noisy measurements}
In this setting each sensor is allowed to take $k$ repeated noisy
measurements. Noise independence is assumed between sensors and between
different samples times for each sensor. Then defining
\begin{equation}
  \label{eq-aggregate}
  \textbf{P}_{r_i}(k)= \sum_{l=1}^{k} \frac{P_{r_i}(l)}{k},
\end{equation}
the variance of the noise becomes $\sigma^2 (k) = \frac{\sigma^2}{k}$. The
regional posterior in then becomes
\begin{multline*}
  p(\ln \textbf{P}_{r_i}(k)|H_j) P(H_j) = \frac{1}{A} \int_{ W_j}
  \frac{1}{(2 \pi \sigma^2
    (k))^{1/2}}\\
  \cdot \exp\!\Big(-\frac{\left(\ln \textbf{P}_{r_i}(k)-\ln(\frac{P d_0}{d_0
        +||x_i-y||^{\beta}})\right)^2}{2 \sigma^2 (k)}\Big) dy,
\end{multline*}
and the joint conditional regional posterior becomes:
\begin{multline*}
  p(\ln \textbf{P}_{r_1}(k),\dots,\ln \textbf{P}_{r_N}(k)|H_i) P(H_i) =
  \frac{1}{A} \int_{W_i} dy\\
  \prod_{j=1}^{N} \frac{1}{(2 \pi \sigma^2(k))^{1/2}}
  \exp\!\Big(-\frac{\left(\ln
      \textbf{P}_{r_j}(k)-\ln(\frac{Pd_0}{d_0+||x_j-y||^{\beta}})\right)^2}{2\sigma^2
    (k)}\Big).
\end{multline*}
Note that, as $k~\to~\infty$, the noise variance approaches zero, and the
probability density approaches a delta function.
\begin{remark}\label{Rem-inf-nonoise}
  Let $\delta$ be the Dirac delta function.  In the infinite measurement
  case, $\lim_{k\to\infty}\sigma^2 (k)=0$, and the probability density satisfies
  \begin{align*}
    &p(\ln \textbf{P}_{r_i}|y) \\
    &= \lim_{k \rightarrow \infty} \frac{1}{(2 \pi
      \sigma(k)^2)^{1/2}}\exp\!\Big(-\frac{\left(\ln
        \textbf{P}_{r_i}-\ln(\frac{P
          d_0}{d_0+||x_i-y||^{\beta}})\right)^2}{2 \sigma^2 (k)}\Big) \\
    &= \delta \Big(\ln \textbf{P}_{r_i} - \ln \frac{P
      d_0}{d_0+\|x_i-y\|^{\beta}} \Big). \eqoprocend
  \end{align*}
\end{remark}
\begin{remark}\label{Rem_aggregation}
  In the sequel, for notational simplicity we will treat the aggregated
  measurement case as if it were identical to the single measurement case
  with the caveat that the variance goes to zero. \oprocend
\end{remark}
\subsection{All-to-all information MAP estimation}\label{subsec_CA}
In the all-to-all communication (A2A) case, full information is available. Using the conditional probability in \eqref{eq-joint_cond_region} MAP
selects the hypothesis $H_{i^*}$ according to
\begin{equation}\label{eq-MAP}
i^*= \arg \max_i p(\ln
P_{r_1},\dots,\ln P_{r_n}|H_i) P(H_i).
\end{equation}
Per Remark~\ref{Rem_aggregation},
this selection scheme applies to both the single and the aggregated
measurement cases.

Before we proceed to deriving the results in the next section, we introduce the definition of the Voronoi diagrams.
\vspace{-0.15in}
\begin{definition}[Voronoi Diagrams]\label{Def-Voronoi}
Given $N$ sensors located at positions $\{x_1,\dots,x_N\} \in C$, we define the Voronoi diagram associated with the $i$th sensor, as follows
\begin{equation*}
V_i = \{x \in C : \|x-x_i\| \leq \|x-x_j\|, \forall j \neq i \}.
\end{equation*}
\end{definition}
%
\section{Preliminary properties of regional localization for one and two
  sensors}
\label{sec_properties}
In this section we derive certain geometric properties of MAP estimation as
$k~\to~\infty$ in~\eqref{eq-aggregate}. These geometric properties allow us
to conclude the following two results.  First, for certain source
locations, a single sensor suffices to asymptotically detect the correct
hypothesis.  Second, for the asymptotic detection problem with two sensors,
the selection of Voronoi partitions as localization regions leads to exact
localization.  These $2$ results should be viewed against the fact that,
even in the noise-free case, at least $3$ non-collinear sensors are
needed for exact localization.
In this section we conduct a large sample analysis to prove an interesting
geometric interpretation of the conditional probability densities. This analysis recognizes that when $k~\to~\infty$
in~\eqref{eq-aggregate}, the Gaussian density approaches a Dirac delta
function. Before we state the lemma that captures this property, we
mention a basic property of the Dirac delta function \cite{KH:96}.
\vspace{-0.1in}
\begin{lemma}[On the Dirac delta function]
  If $\map{g}{\real}{\real}$ is differentiable and vanishes at positions
  $x_\gamma$, $\gamma\in\Gamma$, then
  \begin{equation}\label{eq-delta_prop}
    \delta[g(x)]=\sum_{\gamma\in\Gamma} \frac{\delta(x-x_\gamma)}{|g'(x_\gamma)|}.
  \end{equation}
\end{lemma}
In keeping with Remark~\ref{Rem_aggregation}, consider the situation where
\begin{multline*}
  p(\ln P_{r_i} | y \in W_j)P(y \in W_j)  \\
  = \frac{1}{A} \int_{W_j} \delta \left( \ln P_{r_i} - \ln \frac{P
      d_0}{d_0+\|x_i-y\|^{\beta}}\right) dy.
\end{multline*}
Define the circle $C(r,x_i) = \setdef{y\in\real^2}{\|y-x_i\|=r}$ and
denote its intersection with the region $W_j$ by $S(W_j,r,x_i)~=~ C(r,x_i)
\bigcap W_j$.  Clearly, this intersection set $S(W_j,r,x_i)$ is the union
of certain arcs of $C(r,x_i)$. Define $\theta(W_j,r,x_i)$ to be the sum of
the angles subtended by these arcs. Call $H_j =\left( y \in W_j \right)$.  If we let $y=[y_1,y_2]^{T}$ and define
\begin{multline*}
  f(y_1,y_2,P_{r_i})= \ln P_{r_i} - \ln P d_0 \\
  + \ln(d_0+((x_{i_1}-y_1)^2+(x_{i_2}-y_2)^2)^{\beta/2}),
\end{multline*}
then
\begin{align*}
\tiny
  p(\ln P_{r_i} | H_j) P(H_j) = \frac{1}{A} \int_{W_j} \delta
  (f(y_1,y_2,P_{r_i})) dy_2 dy_1.
\end{align*}
We are now ready for the following lemma.
\begin{lemma}[The arc-length property]
  \label{L-arcs_property}
  Given a region $W_j$, the conditional probability density satisfies
  \begin{multline*}
    p(\ln P_{r_i}|y \in W_j) P(y \in W_j)\\
    = \mathop{\frac{1}{A}\int}_{W_j} \delta\Big(\ln P_{r_i} - \ln \frac{P
      d_0}{d_0+\|x_i-y\|^{\beta}}\Big) dy_2 dy_1,
  \end{multline*}
  and, if we let $r_i=(\frac{P}{P_{r_i}}-1)^{\beta}$, then
  \begin{equation*}
    p(\ln P_{r_i}|y \in W_j) P(y \in W_j) = \frac{d_0+r_i^{\beta}}{A\beta
      r_i^{\beta-2}} \theta(W_j,r_i,x_i).
\end{equation*}
\end{lemma}
The proof of the lemma is provided in the appendix.  This lemma can be
interpreted as follows. Asymptotically, $P_{r_i}$ directly provides the
circle of radius $r_i$ centered at $x_i$ where the sensor is located.
$\theta(W_j,r_i,x_i)$ is simply the angle subtended by the intersection of
this circle with $W_j$. The quantity $p(\ln P_{r_i}|y \in W_j) P(y \in
W_j)$ that is used in MAP is proportional to this angle. We describe
further the significance of this result after
Lemma~\ref{L-two_sensors-Vor}.  In the following lemma, we show that having
Voronoi partitions as well as the MAP estimation algorithm, make the
probability of error zero in the two sensors case.
\vspace{-0.1in}
\begin{lemma}[Optimality of Voronoi for $2$ sensors]\label{L-two_sensors-Vor}
  Consider two points $x_1$ and $x_2$ in $C \subset \real^2$. Let $V_1$ and
  $V_2$ be the Voronoi diagrams associated with $x_1$ and $x_2$.  Take $s \in
  C$ and let $r_1 = \|s-x_1\|$ and $r_2=\|s-x_2\|$.  Then as $k$
  in~\eqref{eq-aggregate} tends to infinity, MAP localization algorithm finds the region containing $s$, with zero probability of error with only
  two sensors.
\end{lemma}
\vspace{-0.2in}
\begin{figure}[h]
  \centering
  \resizebox{.5\linewidth}{!}{\input{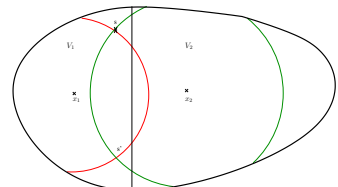}}
  \vspace{-0.1in}
  \caption{\scriptsize This figure shows two nodes $x_1$ and $x_2$, with a source
    $s\in V_1$.}
  \label{fig_two_nodes}
\end{figure}
\begin{proof}
  The proof of this lemma follows the same principle as the proof of Lemma~\ref{L-arcs_property}.
    We point out the differences below. In figure \ref{fig_two_nodes},
  the source $s$ is indicated as one of the two points of intersection of
  the red and green circle.
  The posterior becomes 
$ h(y_1,P_{r_1},x_1,P_{r_2},x_2) =
    \int_{H(a,W_j)} \delta \left(f(a,y_2,P_{r_1},x_1)\right)
    \cdot \delta \left(f(a,y_2,P_{r_2},x_2)\right) dy_2,$
here $P_{r_1}$ and $P_{r_2}$ are the received powers by the sensors.  The integrals are non zero only at the intersections of two sets
which we prove to be nothing but the intersection of the circles in Figure \ref{fig_two_nodes}.
  The definition of Voronoi (Def.~\ref{Def-Voronoi}) implies in the
  two sensors case that if $s' \in C$, then $s \in V_1 \implies s' \in V_1$. In fact $  s \in V_1 \iff \|x_1-s\| \leq \|x_2-s\| \iff \|x_1-s'\| \leq \|x_2-s'\| \iff s' \in V_1,$ if $s' \in C$. The other case to consider is $s' \notin C$. In this case $s$ is the unique point of intersection of the circles in $C$. Both ways $h$ is non-zero only in the correct region.
\end{proof}
The two lemmas presented in this section have interesting implications.
Lemma~\ref{L-arcs_property} implies that, for certain source locations and
as the noise becomes smaller, the MAP estimation algorithm can determine
the correct region containing the source with only one sensor. That is true
when the circle centered at a sensor location with radius $r_i$ is included
in the region $W_j$.  Lemma~\ref{L-two_sensors-Vor} on the other hand,
gives one example where the selection of Voronoi partitions as localization
regions makes it possible to locate the source with only two sensors. This
would not have been the case with two sensors with general convex regions.

In general, the noise will not be vanishing and the decision needs to be
made with a finite number of measurements. We will prove in
Section~\ref{sec_conv} that the algorithms presented in
Section~\ref{sec_map} below converge almost surely as $k~\to~\infty$
in~\eqref{eq-aggregate}.
\section{Decision making algorithms with all-to-all or limited information}
\label{sec_map}
In this section we present two algorithms based on the MAP estimation scheme. The regions can take any shape as long as they are compact. In the limited communication case, sensors can only
talk to their neighbors. We assume there is a communication graph that describes the information exchange among robots. We consider two cases: the all-to-all (A2A) communication case and the limited communication case. In the A2A we adopt the complete undirected communication graph. In the limited communication case we make the following degree assumption: each node has at least two neighbors, i.e., each node appears in at least two edges.
\subsection{All-to-all communication}
In this subsection we present the all-to-all communication algorithm (Algorithm~\#1), where
we apply the classical MAP estimation described in Section~\ref{subsec_CA} on the complete network.
\noindent\rule{1.0\linewidth}{0.001cm} \vspace*{1ex} {\small
 \textbf{Algorithm \#1:} \emph{All-to-all communication MAP}\\
  \textbf{Network:} nodes $\{1,\dots,N\}$ with complete
  communication graph\\
  State of sensor $i$ is $w_i:= \{\ln P_{r_i},x_i,W_i\}$\\
  Sensor $i$ executes
  \begin{algorithmic} [1]
    \STATE \texttt{For} all $j \in \{1,\dots,N\} \setminus \{i\}$
    \STATE \myindent transmit state $w_i$ and receive state $w_j$
    \STATE \myindent calculate $\theta_{j} := p(\ln P_{r_1},\dots,\ln P_{r_N}|H_j)P(H_j)$
    \STATE find $j^{*}:= \arg \max_{j \in \{1,\dots,N\}} \theta_{j}$
    \STATE \textbf{Return:} \texttt{decision} $j^*$
  \end{algorithmic}}
\noindent\rule{1.0\linewidth}{0.001cm}
\vspace{-0.1in}
\begin{assumption}[A2A connectivity and noncollinearity] \label{A-Conn-NonCol-a2a}
We assume that the graph is complete. That is all nodes can communicate with each other. We also assume that at least three sensors in the graph are non-collinear.
\end{assumption}
\subsection{Limited communication}
\label{sub_sec_DA2}
In the limited communication algorithm, each sensor
acquires data from its neighbors and calculates a joint conditional
density \eqref{eq-joint_cond_region}. Each sensor then applies MAP to
choose the most likely hypothesis. In Algorithm~\#2 each sensor adds the
hypothesis of the source being outside its neighborhood, computes the corresponding conditional density and compares it to the densities corresponding to neighboring regions. Let
$\mathcal{N}_i=\{\setdef{j \in \{1,\dots,N\}}{j \textrm{ is a
  neighbor of } i } \union \{i\}\}$ and let $N_i = |\mathcal{N}_{i}|$ be its
cardinality. We described the ``source outside neighborhood'' hypothesis as
hypothesis number $0$ $(H_0)$. We also define $\ln P_{r_{\mathcal{N}_i}}$ to be the vector composed of all measurements $\ln P_{r_i}$ where $i \in \mathcal{N}_i$. Finally, we mention that given an event $H_i = x \in W_i$, the complement event is defined by $\overline{H}_i = x \in W^{C}_i$.
Computing the density in the complement of the neighborhood requires only
the addition of information about the total region. In fact by applying the total probability theorem, we get that
\begin{align*}
\small
   &p(\ln P_{r_j}|H_0) P(H_0)= p(\ln P_{r_j}|\overline{\bigcup_{i \in \mathcal{N}_j} H_i}) P(\overline{\bigcup_{i \in \mathcal{N}_j} H_i})\\
   &=p(\ln P_{r_j|y \in C}) P(y \in C)-\sum_{i \in \mathcal{N}_j} p(\ln P_{r_j}|H_i) P(H_i)
\end{align*}

\noindent\rule{1.0\linewidth}{0.001cm} \vspace*{1ex} {\small
  \textbf{Algorithm \#2:} \emph{Limited communication MAP}\\
  \textbf{Network:} nodes $\{1,\dots,N\}$ with arbitrary
  communication graph\\
  State of sensor $i$ is $w_i:= \{\ln P_{r_i},x_i,W_i\}$\\
  Sensor $i$ executes
  \begin{algorithmic} [1]
    \STATE \texttt{For} all $j \in  \mathcal{N}_i$
    \STATE \myindent transmit state $w_i$ and receive state $w_j$
    \STATE \myindent calculate $\theta_{i,j} := p(\ln P_{r_{\mathcal{N}_i}}|H_j)P(H_j)$
    \STATE \myindent calculate $\theta_{i,0}= p(\ln
    P_{r_{\mathcal{N}_i}}|\; H_0) P(H_0)$
    \STATE find $j^{*}:= \arg \max_{i \in \mathcal{N}_i \union \{0\}} \theta_{i,j}$
    \STATE \textbf{Return:} \texttt{decision} $j^*$
  \end{algorithmic}}
\noindent\rule{1.0\linewidth}{0.001cm}
\begin{assumption}[Limited conn. and non-collinearity] \label{A-Conn-NonCol-lim}
We assume that all nodes can communicate to their neighbors. We also assume that
at least three sensors are non-collinear in each neighborhood.
\end{assumption}
\section{Convergence of algorithms}\label{sec_conv}
In this section we prove that the algorithms presented above give the
correct decision $w.p. 1$ with three or more non-collinear sensors. In
keeping with Remark~\ref{Rem_aggregation}, we examine
\eqref{eq-joint_cond_region} as $\sigma~\to~0^+$. We start by stating the
following result.
\vspace{-0.15in}
\begin{lemma}[Property of non-collinear sensors] \label{L-uniqueness_conc}
  For $d_0>0$ and $\beta>0$, given a source $s\in\real^2$ and three
  non-collinear sensors $x_1$, $x_2$ and $x_3\in\real^2$, define the
  function $\map{f}{\real^2}{\real}$ by $f(z) = \sum_{i=1}^{3}\left( \ln
      \frac{d_0+\|z-x_i\|^{\beta}}{d_0+\|s-x_i\|^{\beta}} \right)^2$.
  The function $f(z)$ vanishes if and only if $z = s$.
\end{lemma}
\begin{proof}
  In fact, it is easy to check that the sum is zero at $z=s$.  Uniqueness
  of this solution is verified by noting that the sum of the square terms
  is zero only if all the summands are zero. For that to be true one needs
  to find $z=(x,y)$ such that $ ((x-x_{i1})^2+(y-x_{i2})^2)^{\beta/2} =
  ((s_1-x_{i1})^2+ (s_2-x_{i2})^2)^{\beta/2} \doteq (r_{i}^2)^{\beta/2}, $
  for $i \in \{1,2,3\}$. Expanding and subtracting, we obtain that a
  necessary and sufficient condition for
uniqueness  of the solution is that the three points are non-collinear.
\end{proof}
In fact given non-collinear $x_i$, $i \in \{1,\dots,N\}$ with $N \geq 3$, a
compact region $W_j$ and a source $s \notin W_j$, there always
exists
\begin {equation*} 
  D_j = \min_{y \in W_j} \|s-y\| = \operatorname{dist}(W_j,s),
\end{equation*}
such that $D_j > 0$. Also for a compact regions $W_j$, there exists $U_j>0$ such that
\begin{equation} \label{eq-ubound1}
  U_j \geq \max_{y \in W_j,i \in \{1,\dots,N\}} \left|\ln
    \frac{d_0+\|y-x_i\|^{\beta}}{d_0+\|s-x_i\|^{\beta}}\right|.
\end{equation}
Because of the non-collinearity and the fact that $D_j>0$, there
exists an $L_j>0$ such that
\begin{equation} \label{eq-lbound}
  L_j \leq \min_{y \in W_j} \sum_{i=1}^{N}\left( \ln
    \frac{d_0+\|y-x_i\|^{\beta}}{d_0+\|s-x_i\|^{\beta}} \right)^2.
\end{equation}
This follows from the lemma above (i.e., from the fact that the sum has a
unique global minimum at $s$). Define
\begin{equation} \label{eq-delta}
  \eta_j = \sqrt{U_j^2+\frac{L_j}{\alpha N}} - U_j,
\end{equation}
where $\alpha > 1$, then we have the following result.
\begin{lemma} \label{L-sum_bound}
  Consider $L_j$, $U_j$ and $\eta_j$ as defined in~\eqref{eq-ubound1}, \eqref{eq-lbound}
  and~\eqref{eq-delta}. Suppose the source $s$ is not in region $W_j$, and $|n_i| \leq
  \eta_j$ for all $i\in\until{N}$. Then
  \begin{equation} \label{eq-exponent_bound1}
    \frac{\alpha-1}{\alpha}L_j \leq \min_{y \in W_j} \sum_{i=1}^{N} \left( \ln
      \frac{d_0+\|y-x_i\|^{\beta}}{d_0+\|s-x_i\|^{\beta}}+n_i \right)^2.
  \end{equation}
\end{lemma}
\begin{proof}
In fact, the sum in \eqref{eq-exponent_bound1} satisfies:
  \begin{align*}
    \sum_{i=1}^{N} & \Big(   \ln \frac{d_0+\|y-x_i\|^{\beta}}{d_0+\|s-x_i\|^{\beta}}+n_i
    \Big)^2 \\
    &\geq L_j - 2 U_j N \eta_j - N \eta_j^2 = L_j + 2 U_j^2 N \\
    &\quad- 2 U_j N
    \sqrt{U_j^2+\frac{L_j}{\alpha N}} -N(U_j^2+\frac{L_j}{\alpha N})
    -NU_j^2 \\
    &\quad+2 N U_j \sqrt{U_j^2+\frac{L_j}{\alpha N}}  =
    \frac{\alpha-1}{\alpha} L_j.
  \end{align*}
\end{proof}

For simplicity of notation, we choose $\alpha = 2$ from here on.  Next, we
introduce the final intermediate lemma before the main results of this
section.
\begin{lemma}[Upper bound for wrong hypothesis] \label{L-exponent_lemma}
  Given $L_j$, $U_j$ and $\eta_j$ as defined in \eqref{eq-ubound1}, \eqref{eq-lbound}
  and \eqref{eq-delta}, define
  \begin{multline*}
    I_j = p(\ln P_{r_1},\dots,\ln P_{r_N}|H_j)P(H_j) = \frac{1}{A (2 \pi \sigma^2)^{N/2}} \\
    \cdot\int_{W_j} \exp\!\Big(-\frac{\sum_{i=1}^{N} \left( \ln P_{r_i}-\ln \frac{P
          d_0}{d_0+\|y-x_i\|^{\beta}} \right)^2}{2 \sigma^2}\Big)  dy.
  \end{multline*}
  If $|n_i| \leq \eta_j$ for $i\in\until{N}$, then for $j\in\until{N}$
  \begin{equation*} 
    I_j \leq \frac{A_j \exp\!\big(-L_j/4 \sigma^2\big)} {A(2 \pi \sigma^2)^{N/2}}.
  \end{equation*}
\end{lemma}
\begin{proof}
  Because of the equality
  \begin{equation*}
    \ln P_{r_i} - \ln \frac{P d_0}{d_0+\|y-x_i\|^{\beta}}= \ln
    \frac{d_0+\|y-x_i\|^{\beta}}{d_0+\|s-x_i\|^{\beta}}+n_i,
  \end{equation*}
  the result directly follows from Lemma~\ref{L-sum_bound} and from the
  fact that the surface integral of a function $f$ is upper bounded by the
  surface integral of a constant function $g$, where $g$ takes
  the maximum value of $f$.
\end{proof}
We are now ready for the convergence theorem. As usual, we define the
$Q$-function $\map{Q}{\real}{\real_{>0}}$ by
\begin{equation*}
  Q(x)=\frac{1}{\sqrt{2\pi}}\int_x^{+\infty} \exp(-y^2/2)dy.
\end{equation*}
\begin{theorem}[Elimination property of wrong hypothesis] \label{theorem_wrong_hypoth}
  Consider $x_i$ non-collinear sensors, $i \in \{1,\dots,N\}$ with $N \geq
  3$.  Let $\sigma$ be the noise variance. Given a source $s \notin W_j$,
  then we have
  \begin{equation*}
    \Prob \biggl[ p(\ln P_{r_1},\dots,\ln P_{r_N}|H_j) P(H_j) \leq
    \epsilon_j(\sigma) \biggr]
    \geq \mu_j(\sigma),
  \end{equation*}
  where
  \begin{equation*}
    \epsilon_j(\sigma) = \frac{A_j \exp(-L_j/4 \sigma^2)}{A (2 \pi
      \sigma^2)^{N/2}}, \quad
    \mu_j(\sigma) = (1-2 Q(\eta_j/\sigma))^{N}.
  \end{equation*}
  Furthermore, as $\sigma\to0^+$, we have $\epsilon_j(\sigma)~\to~0^+ \quad\text{and}\quad
    \mu_j(\sigma)~\to~1^-$.
\end{theorem}
\begin{proof}
  From Lemmas~\ref{L-sum_bound} and~\ref{L-exponent_lemma} , we have that
  \begin{align*}
    \Prob&\big[ p(\ln P_{r_1},\dots,\ln P_{r_N}|H_j) P(H_j) \leq
    \epsilon_j(\sigma) \big] \\
    &\geq \Prob \left[ [n_1,\dots,n_N]^{T} \in [-\eta_j,\eta_j]^{N} \right] \\
    &= \prod_{i=1}^{N} \biggl( \frac{1}{2}-\Prob[n_i>\eta_j]
    +\frac{1}{2}-\Prob[n_i<-\eta_j] \biggr)
    \\
    &= \big( 1-2Q(\eta_j/\sigma) \big)^{N}.
  \end{align*}
  The first inequality comes from the fact that Lemmas~\ref{L-sum_bound}
  and~\ref{L-exponent_lemma} hold whenever all $|n_i| \leq \eta_j$. The
  proofs of the two limits of $\epsilon_j$ and $\mu_j$ are immediate.
\end{proof}
This theorem states that, as $\sigma~\to~0^+$, the probability that the
joint density function $p(\ln P_{r_1},\dots,\ln P_{r_N}|H_j) P(H_j)$ takes
an arbitrarily small value goes arbitrarily close to $1$ when $H_j$ is not
the correct hypothesis.
This is so as $Q(x)~\to~0$ as $x~\to~\infty$.  To complement the
Theorem~\ref{theorem_wrong_hypoth}, we prove below that for
the \emph{correct} hypothesis, the probability density will be lower
bounded by a positive term $w.p.1$.
\begin{theorem}[Strict positivity for correct hypothesis]
  \label{theorem_correct_hypothesis}
  Consider $x_i$ non-collinear sensors, $i \in \{1,\dots,N\}$ with $N \geq
  3$.  Let $\sigma$ be the noise variance. If $s \in W_i$, then we have
  \begin{equation*}
    \Prob \left[ p(\ln P_{r_1},\dots,\ln P_{r_N}|H_i) P(H_i) \geq
      \Psi(\sigma) \right] \geq \Omega(\sigma),
  \end{equation*}
  where
  \begin{align*}
    &\Psi(\sigma) = p(\ln P_{r_1},\dots,\ln P_{r_N}) -
    \sum_{\substack{{j=1,\dots,N} \\{j \neq i}}} \frac{A_j \exp(-L_j/4
      \sigma^2)}{A (2 \pi  \sigma^2)^{N/2}}, \\
    &\Omega(\sigma) = \prod_{\substack{{j=1,\dots,N} \\{j \neq i}}}
     \mu_j(\sigma) =\prod_{\substack{{j=1,\dots,N} \\{j \neq i}}} (1-2
    Q(\eta_j/\sigma))^{N}.
  \end{align*}
  Furthermore, as $\sigma~\to~0^+$, we have $\Psi(\sigma)~\to~p(\ln P_{r_1},
  \dots,\ln P_{r_N})>0  \quad\text{and}\quad
    \Omega(\sigma)~\to~1^-$.
\end{theorem}
\begin{proof}
  The proof of this theorem follows directly from
  Theorem~\ref{theorem_wrong_hypoth} and the total probability theorem.
  Call $~z=~[\ln P_{r_1},\dots,\ln P_{r_N}]^T$. We know from the total
  probability theorem that
  \begin{align*}
    p(z) &=\sum_{j=1}^{N} p(z|H_j)P(H_j) \\
    &= p(z|H_i) P(H_i)
    + \sum_{\substack{{j=1,\dots,N} \\{j \neq i}}} p(z|H_j)P(H_j)
  \end{align*}
  and, in turn, that
  \begin{equation*}
    p(z|H_i) P(H_i) = p(z) - \sum_{\substack{{j=1,\dots,N}
        \\{j \neq i}}} p(z|H_j)P(H_j).
  \end{equation*}
  From Theorem~\ref{theorem_wrong_hypoth}
  \begin{align*}
    &\Prob \left[ p(z|H_i) P(H_i) \geq
      p(z) -\sum_{\substack{{j=1,\dots,N} \\{j \neq i}}} \epsilon_j(\sigma)
       \right] \nonumber  \\
    &\geq \prod_{\substack{{j=1,\dots,N} \\{j \neq i}}} \Prob
     \biggl[ p(z|H_j) P(H_j) \leq \epsilon_j(\sigma) \biggr]
    \geq \prod_{\substack{{j=1,\dots,N} \\{j \neq i}}} \mu_j(\sigma)
\end{align*}
As $\sigma~\to~0^+$, $\Psi(\sigma)~\to~p(z)$ and $\Omega(\sigma)~\to~1^-$.
\end{proof}
Theorem~\ref{theorem_correct_hypothesis} complements
Theorem~\ref{theorem_wrong_hypoth} in that is shows that as $\sigma~\to~0^+$,
the probability density conditioned on the correct hypothesis is
lower bounded by a strictly positive term. This event happens
asymptotically with probability $1$.

Under Assumption~\ref{A-Conn-NonCol-a2a}, as MAP follows \eqref{eq-MAP},
Theorems~\ref{theorem_correct_hypothesis} and~\ref{theorem_wrong_hypoth}
 complete the proof of $a.s$ convergence of the all-to-all communication
  MAP algorithm. Similarly for the limited communication, under
  Assumption~\ref{A-Conn-NonCol-lim}, MAP estimation converges almost
  surely when applied to regional localization.
\section{Simulations} \label{sec_simulations}
In this section we show simulation results illustrating the type of
decision obtained by our algorithms. We also show a comparison plot between
the all-to-all and the limited information algorithms.
\vspace{-0.1in}
\begin{figure}[h]
  \centering
  \includegraphics[scale=0.12]{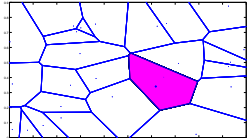}
    \vspace{-0.2in}
  \caption{\scriptsize This figure shows the all-to-all communication algorithm with 25
    sensors and one source whose region was correctly detected.}
  \label{fig_cent_corr}
\end{figure}
The plot in Figure \ref{fig_cent_corr} shows a correct detection of the
source. The shaded region corresponds to the one detected by the algorithm,
the source is shown as a star and the sensors as the dots.
\begin{figure}[thpb]
  \centering
  \includegraphics[scale=0.17]{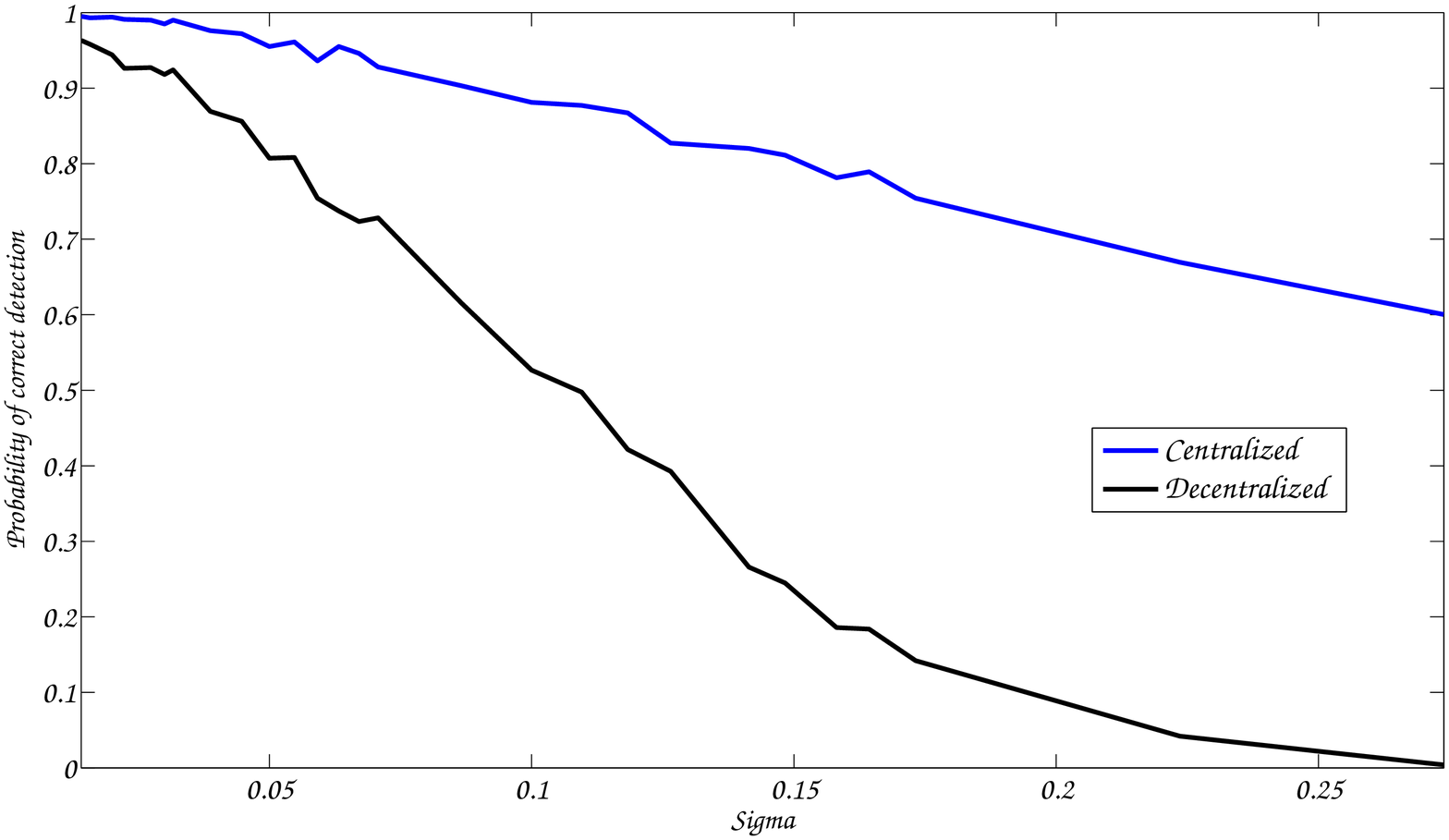}
  \vspace{-0.1in}
  \caption{\scriptsize This figure shows the probability of correct decision making of
    the all-to-all and limited communication cases over $1000$ runs, with $20$ sensors.
    The blue curve corresponds to the A2A communication algorithm,
    and the black curve corresponds to the limited
     communication algorithm (with the majority vote).}
  \label{fig_cent_null}
\end{figure}
Figure \ref{fig_cent_null} shows the results obtained from batches of
$1000$ runs, with $N=20$ sensors. The plots are obtained as follows: If
$d_j$ is the index of the region chosen by sensor $j$, we compare it to the
correct index $d^*$. For the comparison to be fair, and since the
limited communication algorithm makes decisions only about its neighborhood,
we compare the two algorithms by adding a communication round to the limited communication algorithm, where
we look at all the sensors that decided that the source is in their neighborhood, noting that
the decisions could be inconsistent, we run a majority vote among the aforementioned sensors, and
compare the final decision to the correct one.
The black curve
in Figure~\ref{fig_cent_null} represents the
limited communication decisions, while the blue curve represents the
all-to-all communication decision. Figure~\ref{fig_cent_null} shows the
probability of making a correct decision as $\sigma$
increases. 
It is not surprising that as the noise variance increases the
probability of correct decision making decreases.
%
\section{Conclusion} \label{sec_conc}
In this paper we have presented an all-to-all and a limited communication
algorithm based on MAP that succeed in identifying a region that contains a
source.  We have also presented an asymptotic analysis and derived some
geometric properties of our algorithms. Those properties had the implication that
for certain source positions, it is possible to solve the regional
localization problem with a unique sensor. We showed that choosing Voronoi
partitions for localization regions asymptotically achieves zero
probability of error in the two sensors case. We were also able to prove that
when readings from three or more non-collinear sensors are available, the algorithms
choose the correct region almost surely. As an extension to this work we are studying
 how to optimally position the sensors, and choose the localization regions,
 so that we minimize the probability of error of our algorithms.
\vspace{-0.15in}
\bibliographystyle{ieeetr}
\bibliography{FB,alias,New,Main,Sandra}

\appendix
\begin{proof}[Proof of Lemma~\ref{L-arcs_property}]
  Let
  \begin{equation*}
    H(a,W_j) = \left( y_2 \in \real \textrm{ such that given } y_1=a,
      [a,y_2] \in W_j \right)
  \end{equation*}
  \vspace{-0.1in}
  \begin{multline}\label{eq-hy1}
    h(y_1,P_{r_i},x_i) =\\
    \int_{H(a,W_j)}\delta\left(\ln P_{r_i} - \ln
      \frac{P d_0}{d_0+ \|x_i-y\|^{\beta}}\right)dy_2\\
    = \int_{H(a,W_j)} \delta \left(f(a,y_2,P_{r_i},x_i)\right) dy_2.
  \end{multline}
Since \hspace{-0.3in}
\begin{multline*}
\frac{d}{d y_2} f(y_1,y_2,P_{r_i},x_i) = f'(y_1,y_2,P_{r_i},x_i)\\
= \frac{\beta}{2} \cdot 2 \cdot (-1) \cdot  (x_{i_2}-y_2)
\frac{\left((x_{i_1}-y_1)^2+(x_{i_2}-y_2)^2\right)^{\frac{\beta}{2}-1}}{d_0
+\left((x_{i_1}-y_1)^2+(x_{i_2}-y_2)^2\right)^{\frac{\beta}{2}}}
\end{multline*}
If we fix $y_1 = a$, we can solve for $y_2(a)$ such that,
$f(a,y_2(a),P_{r_i},x_i)=0$. In fact
\begin{multline} \label{eq-circle_eq}
f(a,y_2(a),P_{r_i},x_i) = 0  \\
\Leftrightarrow \ln P_{r_i}-\ln \frac{P
d_0}{d_0+\left((x_{i_1}-a)^2+(x_{i_2}-y_2(a))^2\right)^{\frac{\beta}{2}}}= 0 \\
 \Leftrightarrow (x_{i_1}-a)^2+(x_{i_2}-y_2(a))^2 =
\left(\frac{P-P_{r_i}}{P_{r_i}}d_0\right)^{\frac{2}{\beta}} = r_i^2,
\end{multline}
where $r_i =
\left((\frac{P}{P_{r_i}}-1)d_0\right)^{\frac{1}{\beta}}$. Observe
$H(a,W_j)$ has at most two elements satisfying equation
\eqref{eq-circle_eq}, one or both of:
\begin{equation}
y_{2,1}(a) = x_{i_2} - \sqrt{r_i^2-(x_{i_1}-a)^2} \label{eq-y2_a}
\end{equation}
or,
\begin{equation}
y_{2,2}(a) = x_{i_2} + \sqrt{r_i^2-(x_{i_1}-a)^2},\label{eq-y'2_a}
\end{equation}
whenever $r_i^2 \geq (x_{i_1}-a)^2$. Using property
\eqref{eq-delta_prop} of the dirac delta function, and substituting
with $y_{2,1}(a)$ and $y_{2,2}(a)$ obtained in \eqref{eq-y2_a} and
\eqref{eq-y'2_a}, \eqref{eq-hy1} becomes:
\begin{equation*}
h(a,P_{r_i},x_i) = \int_{H(a,j)} \delta(f(a,y_2,P_{r_i}),x_i) dy_2,
\end{equation*}
takes the values
\begin{displaymath}
\left\{
\begin{array}{ll}
\int_{H(a,W_j)}
\frac{\delta(y_2-y_{21}(a))}{|f'(a,y_2,P_{r_i},x_i)|} dy_2 \\
\textrm{if $y_{2,1}(a) \in H(a,W_j)$ but $y_{2,2}(a) \notin H(a,W_j)$} \\\int_{ H(a,W_j)}\frac{\delta(y-y_{2,2}(a))}{|f'(a,y_2,P_{r_i},x_i)|}
\ dy_2 \\ \textrm{if $y_{2,2}(a) \in H(a,W_j)$ but
$y_{21}(a) \notin H(a,W_j)$} \\
\int_{H(a,W_j)}
\frac{\delta(y_2-y_{2,1}(a))}{|f'(a,y_2,P_{r_i},x_i)|} +
\frac{\delta(y_2-y_{2,2}(a))}{|f'(a,y_2,P_{r_i},x_i)|} \ dy_2 \\ \textrm{if both $y_{2,1}(a)$ and $y_{2,2}(a) \in H(a,W_j)$}
\end{array} \right.
\end{displaymath}
Define $I_1(a,W_j)$, the indicator function satisfying
\begin{displaymath}
I_1(a,W_j) =\left\{ \begin{array}{ll}
1 & \textrm{if $y_{2,1}(a) \in H(a,W_j)$}  \\
0 & \textrm{otherwise} \end{array} \right.
\end{displaymath}
Similarly define $I_2(a,W_j)$, the indicator function satisfying
\begin{displaymath}
I_2(a,W_j) =\left\{ \begin{array}{ll}
1 & \textrm{if $y_{2,2}(a) \in H(a,W_j)$} \\
0 & \textrm{otherwise} \end{array} \right.
\end{displaymath}
Then, \eqref{eq-hy1} becomes
\begin{multline*}
h(a,P_{r_i},x_i)= \frac{1}{|f'(a,y_{2,1}(a),P_{r_i},x_i)|}I_1(a,W_j) \\
+\frac{1}{|f'(a,y_{2,2}(a),P_{r_i},x_i)|}I_2(a,W_j).
\end{multline*}
By substituting from \eqref{eq-y2_a},  we get
\begin{align*}
\tiny
&\frac{1}{|f'(a,y_2(a),P_{r_i},x_i)|} \\
=&\frac{d_0+\left((x_{i_1}-a)^2+r_i^2-(x_{i_1}-a)^2\right)^{\frac{\beta}{2}}}{\beta
\sqrt{r_i^2-(x_{i_1}-a)^2} \left((x_{i_1}-a)^2+r_i^2-(x_{i_1}-a)^2
\right)^{\frac{\beta}{2}-1}} \\
=& \frac{d_0+r_i^{\beta}}{\beta \sqrt{r_i^2-(x_{i_1}-a)^2}}
\frac{1}{r_i^{\beta-2}} = \frac{d_0+r_i^{\beta}}{\beta r_i^{\beta-2}} \cdot
\frac{1}{\sqrt{r_i^2-(x_{i_1}-a)^2}}
\end{align*}
Let $\mathcal{C}_j=\{ x \in \real \textrm{ such that }
(x,y_{2,1}(x)) \in W_j\}$ and $\mathcal{C'}_j=\{ x \in \real
\textrm{ such that } (x,y_{2,2}(x)) \in W_j\}$. Note that
\begin{equation*}
x \in \mathcal{C}_j \Rightarrow I_1(x,W_j)=1 \textrm{ and } x \in
\mathcal{C'}_j \Rightarrow I_{2}(x,W_j)=1.
\end{equation*}
Then,
\begin{align} \label{eq-P_afixed-2sets} 
\tiny
p&(\ln P_{r_i}|y \in W_j) P(y \in W_j) = \frac{1}{A} \int_{\mathcal{C}_j \bigcup \mathcal{C'}_j} h(y_1,P_{r_i},x_i) dy_1 \nonumber\\
&=\frac{1}{A} \biggr(  \int_{\mathcal{C}_j} h(y_1,P_{r_i},x_i) dy_1 +
\int_{\mathcal{C'}_j}
h(y_1,P_{r_i},x_i) dy_1 \biggl)\nonumber \\
&= \frac{1}{A} \int_{\mathcal{C}_j} \frac{1}{|f'(y_1,y_{2,1}(y_1),P_{r_i},x_i)|}
dy_1 \nonumber \\
&+ \frac{1}{A} \int_{\mathcal{C'}_j}
\frac{1}{|f'(y_1,y_{2,2}(y_1),P_{r_i},x_i)|} dy_1.
\end{align}
Write
\vspace{-0.15in}
\begin{align}
\tiny
&\mathcal{C}_j =
\bigcup_{\alpha=1}^{s} A_{\alpha} \textrm{ , with } \bigcap_{\alpha}
A_{\alpha} = \emptyset\textrm{ , and } A_{\alpha} =
[a_{1_\alpha},a_{2_\alpha}] \label{eq-C_A1} ,\\
&\mathcal{C'}_j = \bigcup_{\alpha=1}^{s'} A'_{\alpha} \textrm{ , with
} \bigcap_{\alpha} A'_{\alpha} = \emptyset\textrm{ , and }
A'_{\alpha} = [a'_{1_\alpha},a'_{2_\alpha}]\label{eq-C_A2}.
\end{align}
Equation \eqref{eq-P_afixed-2sets} can then be written as the sum
\begin{align*}
\tiny
&A \cdot p(\ln P_{r_i}|y \in W_j) P(y \in W_j)=\\
&\sum_{\alpha=1}^{s}
\int_{A_{\alpha}} \frac{1}{|f'(y_1,y_{2,1}(y_1),P_{r_i},x_i)|}dy_1\\
&+\sum_{\alpha=1}^{s'} \int_{
A'_{\alpha}} \frac{1}{|f'(y_1,y_{2,2}(y_1),P_{r_i},x_i)|}y_1.
\end{align*}
\vspace{-0.2in}
\begin{multline*}
\tiny
= \sum_{\alpha=1}^{s} \int_{a_{2_\alpha}}^{a_{1_\alpha}}
\frac{1}{|f'(y_1,y_{2,1}(y_1),P_{r_i},x_i)|} dy_1\\
+\sum_{\alpha=1}^{s'}
\int_{a'_{1_\alpha}}^{a'_{2_\alpha}} \frac{1}{|f'(y_1,y_{2,2}(y_1),P_{r_i},x_i)|} \ dy_1,
\end{multline*}
\vspace{-0.2in}
\begin{multline*}
\tiny
= \sum_{\alpha=1}^{s} \int_{a_{2_\alpha}}^{a_{1_\alpha}}
\frac{d_0+r_i^{\beta}}{\beta r_i^{\beta-2}} \cdot
\frac{1}{\sqrt{r_i^2-(x_{i_1}-y_1)^2}} \ dy_1\\
+ \sum_{\alpha=1}^{s'} \int_{a'_{2_\alpha}}^{a'_{1_\alpha}}
\frac{d_0+r_i^{\beta}}{\beta r_i^{\beta-2}} \cdot
\frac{1}{\sqrt{r_i^2-(x_{i_1}-y_1)^2}} \ dy_1\\
= \sum_{\alpha=1}^{s}
\frac{d_0+r_i^{\beta}}{\beta r_i^{\beta-2}} \cdot \arctan \frac{x_{i_1}-a}{\sqrt{r_i^2-(x_{i_1}-y_1)^2}}|_{a_{1_\alpha}}^{a_{2_\alpha}}\\
+ \sum_{\alpha=1}^{s'} \frac{d_0+r_i^{\beta}}{\beta r_i^{\beta-2}}
\cdot \arctan
\frac{x_{i_1}-y_1}{\sqrt{r_i^2-(x_{i_1}-y_1)^2}}|_{a'_{1_\alpha}}^{a'_{2_\alpha}}
\end{multline*}
Note that 
\begin{multline*}
\tiny
\arctan\frac{x_{i_1}-a}{\sqrt{r_i^2-(x_{i_1}-a)^2}} \\=
\arctan\frac{x_{i_1}-a}{\sqrt{(x_{i_1}-a)^2+(x_{i_2}-y_2(a))^2-(x_{i_1}-a)^2}}\\
= \arctan\frac{x_{i_1}-a}{x_{i_2}-y_2(a)} = \frac{\pi}{2}
-\arctan\frac{x_{i_2}-y_2(a)}{x_{i_1}-a}.
\end{multline*}
The conditional probability density becomes after simplifications,
\hspace{-0.1in}
\begin{multline*}
\tiny
p(\ln P_{r_i}|y \in W_j) P(y \in W_j)= \frac{d_0+r_i^{\beta}}{A \beta
r^{\beta-2}} \cdot \sum_{\alpha=1}^{s} \\
\left(\frac{\pi}{2} -
\arctan\frac{x_{i_2}-y_2(a_{2_\alpha})}{x_{i_1}-a_{2_\alpha}} -
\frac{\pi}{2} + \arctan
\frac{x_{i_2}-y_2(a_{1_\alpha})}{x_{i_1}-a_{1_\alpha}} \right) \\
+ \frac{d_0+r_i^{\beta}}{A \beta r_i^{\beta-2}}
\sum_{\alpha=1}^{s'} \left(\frac{\pi}{2}- \arctan
\frac{x_{i_2}-y_2(a'_{2_\alpha})}{x_{i_1}-a'_{2_\alpha}} - \frac{\pi}{2} \right.\\
\left.  + \arctan \frac{x_{i_2}-y_2(a'_{1_\alpha})}{x_{i_1}-a'_{1_\alpha}} \right)= \frac{d_0+r_i^{\beta}}{A \beta r_i^{\beta-2}} \left(
\sum_{\alpha=1}^{s} \theta_{\alpha} + \sum_{\alpha=1}^{s'}
\theta'_{\alpha} \right),
\end{multline*}
where $\theta_{\alpha}$ and $\theta'_{\alpha}$ are the angles of the
arcs in $S(W_j,r_i,x_i)$ described on distinct supports as in
\eqref{eq-C_A1} and \eqref{eq-C_A2} when applicable.
\end{proof}
 \end{document}